\documentclass[12pt]{article}
\usepackage{amssymb,url}
\newtheorem{theo}{Theorem}[section]
\newtheorem{prop}[theo]{Proposition}
\newtheorem{lemma}[theo]{lemma}
\newtheorem{cor}[theo]{Corollary}

\newtheorem{unnumber}{}

\newtheorem{prepf}[unnumber]{Proof}
\newenvironment{proof}{\prepf\rm}{\endprepf}

\newcommand{\qed}{\qquad$\Box$}

\begin{document}

\title{Graphs defined on algebras}
\author{Peter J. Cameron, University of St Andrews}
\date{}
\maketitle

\begin{abstract}
There has been a great deal of attention recently to graphs whose vertex set is a group, defined using the group structure. (The \emph{commuting graph}, where
two elements are joined if they commute, is the oldest and most famous
example.)

The purpose of this paper is to investigate extending the definitions of such
graphs to general algebras (in the sense of universal algebra). It seems
unlikely that such a definition can be made for the commuting graph, or for
various others such as the nilpotency and Engel graphs. However, for graphs
whose definition depends on the notion of subgroup or subalgebra generated
by a subset, the existing definitions
work without change. These graphs include several well-studied examples: the
\emph{power graph}, \emph{enhanced power graph}, \emph{generating graph},
\emph{independence graph}, and \emph{rank graph}. In these cases, some 
results about groups extend to arbitrary algebras unchanged, but others require
specific properties of groups, and pose a challenge to researchers.

In the next two sections, I will describe some extensions to directed graphs
(the directed power graph and the endomorphism digraph) and to simplicial
complexes (the independence and strong independence complexes). The final
section gives explicit descriptions of all of these objects for independence
algebras.
\end{abstract}

\medskip



\section{Introduction}

This paper could be summarised as: Here is some interesting combinatorics
related to finite groups; can it be extended to wider classes of algebras?

Graphs and groups have a long common history, Cayley graphs having been 
introduced in the nineteenth century; these graphs are the basic objects in
geometric group theory as well as a central topic in algebraic combinatorics.
But my topic is not Cayley graphs, which depend on both the group structure
and on the choice of a subset of the group as ``connection set''.

Rather, I am concerned with graphs whose vertex set is a group, or possibly
some ``natural'' (isomorphism-invariant) subset or partition of the group,
and whose joining relation depends only on algebraic properties of the group.
(This implies that the automorphism group of the group acts as automorphisms
of the graph.)

The oldest example of this appears in the paper of Brauer and 
Fowler~\cite{commuting} in 1955. They showed that a simple group of even order
has order bounded by a function of the order of an involution
centraliser. This could be regarded as the first step towards the
classification of finite simple groups. Though graphs are not mentioned
explicitly in the paper, the arguments involve bounding the diameter of the
\emph{commuting graph}, whose vertices are the non-identity elements, two
vertices adjacent of they commute.

The commuting graph of a group has received a lot of attention, and the study
has recently been extended to semigroups~\cite{abk,paulista}. But it does
require, at a minimum, the existence of a distinguished binary operation.
Other graph types discussed in \cite{gog} include the nilpotency,
solubility and Engel graphs, which are specific to groups.

However, there are other graphs defined on groups whose definition depends
only on the notions of subgroup generated by a subset or of endomorphisms,
and these extend naturally to any algebra (in the sense of universal algebras,
that is, sets with prescribed families of operations of various arities).
These include the power graph and two variants, generating graph, 
independence graph, and rank graph, which will be discussed below. The next
two sections discuss directed graphs and simplicial complexes, and the final
section gives details of these structures on independence algebras.

Since group theory is the richest part of abstract algebra, we begin with
graphs on groups and ask which of their properties can be extended to
wider classes of algebras.

All groups and algebras considered here are finite. There are some results
for infinite groups and graphs but the theory is less well developed.

\section{Preliminaries}

\subsection{Graph classes}

Given various types of graphs defined on algebras, there are two main
ways in which we can specify a class of algebras:
\begin{itemize}
\item we could specify two types of graph $\Gamma_1$ and $\Gamma_2$, and ask:
for which algebras $A$ does $\Gamma_1(A)=\Gamma_2(A)$?
\item we could specify a class $\mathcal{C}$ of graphs, and a graph type
$\Gamma$ on algebras, and ask: for which algebras $A$ does $\Gamma(A)$ belong
to $\mathcal{C}$?
\end{itemize}

For the second point, there are several important classes of graphs, closed
under taking induced subgraphs, and hence having a characterisation by
forbidden induced subgraphs. (An \emph{induced subgraph} of a graph is obtained
by taking a suset of the vertex set together with all the edges it contains.)
In most cases, these also have a more structural description. Here are some of
those I will consider.

\medskip\noindent
\textbf{Perfect graphs:} A graph $\Gamma$ is perfect if every induced
subgraph has clique number equal to chromatic number. (The clique number is
the largest number of vertices in a complete subgraph, and the chromatic
number is the smallest number of colours required to colour the vertices so
that adjacent vertices have different colours.) By the \emph{Strong Perfect
Graph Theorem}~\cite{crst}, a graph $\Gamma$ is perfect if and only if it does
not contain a cycle of length $n$, where $n$ is odd and greater than $3$, or
the complement of one, as an induced subgraph.

\medskip\noindent
\textbf{Chordal graphs:} A graph $\Gamma$ is chordal if every cycle of
length greater than $3$ has a chord; that is, cycles of length $n>3$ are
forbidden induced subgraphs.

\medskip\noindent
\textbf{Cographs:} This class has many characterisations, and has been
rediscovered (and given different names) several times. A graph $\Gamma$ is a
\emph{cograph} (or \emph{complement-reducible graph}) if for every induced
subgraph $\Delta$, either $\Delta$ or its
complement is disconnected. A graph is a cograph if and only if it does not
contain the $4$-vertex path as an induced subgraph.

\medskip\noindent
\textbf{Split graphs:} A graph is \emph{split} if and only if the vertex set is
the union of two subsets, one of which induces a complete graph and the other
a null graph. A graph is split if and only if it does not contain the $4$- or
$5$-vertex cycle or two disjoint edges as induced subgraph.

\medskip\noindent
\textbf{Threshold graphs:} A graph is \emph{threshold} if there are
vertex weights $\mathrm{wt}(v)$ on the vertices and a threshold $t$ such that
$v$ and $w$ are joined if and only if $\mathrm{wt}(v)+\mathrm{wt}(w)\ge t$.
A graph is threshold if and only if it does not contain a path or cycle on
$4$ vertices, or two disjoint edges, as induced subgraph.

\medskip
For other connections between groups and the type of graphs considered here,
see~\cite{gandg}.

\subsection{Algebras}

An \emph{$n$-ary operation} on a set $A$ is a function from $A^n$ to $A$. For
$n=0,1,2$, we say ``nullary'', ``unary'' and ``binary'' respectively. A
nullary operation takes no arguments and returns an element of $A$; elements
which arise in this way are called \emph{constants}.

An algebra (in the sense of universal algebra) is a set with a collection of
operations of various arities. Usually we regard these operations as named,
so we can speak of a collection of algebras with the same operations. Thus,
a group is an algebra with one binary operation (multiplication), one unary
operation (inversion) and one nullary operation (identity).

A subalgebra of an algebra $A$ is a subset of $A$ which is closed under all
the operations, and so is itself an algebra. The intersection of subalgebras
is a subalgebra; so we can define the algebra generated by a subset $S$ of $A$
to be the intersection of all the subalgebras of $A$ containing $S$. It can
alternatively be defined as the smallest subset containing $S$ and closed under
all the operations. We denote it by $\langle S\rangle$. The \emph{rank} of an
algebra $A$ is the cardinality of the smallest generating set; we denote it by
$r(A)$. An algebra is \emph{monogenic} if its rank is $1$. We denote ``$B$ is a
subalgebra of $A$'' by $B\le A$.

An algebra $A$ has a unique minimal subalgebra, the intersection of all the
subalgebras of $A$. It is the subalgebra generated by the constants; we denote
it by $E(A)$. If there are no nullary operations, then $E(A)$ is empty.

There are several properties of algebras which will be used later. An algebra
is \emph{monotonic} if $B_1\le B_2\le A$ implies $r(B_1)\le r(B_2)$. An
algebra is \emph{$1$-monotonic} if its rank~$1$ subalgebras are monotonic
(this means that any nontrivial subalgebra of a rank~$1$ subalgebra has
rank~$1$). Groups have this property since a subgroup of a cyclic group is
cyclic. However, it does not hold in general.

\noindent\textbf{Example} Let $A$ be the semigroup generated by an element
$x$ such that $x^4=x^5$ (so that $x^4$ is a zero). Then the subsemigroup
generated by $x^2$ and $x^3$ has rank~$2$ since neither of these elements
is a power of the other.

\subsection{Graphs on groups and algebras}

Here I define various graphs which have been defined on groups in the
literature, but whose definition extends to algebras.

\medskip\noindent
\textbf{The power graph} This was defined by Kelarev and Quinn~\cite{kq}.
The power graph of a group $G$ has vertex set $G$, with an edge joining $x$
and $y$ if either $x$ is a power of $y$, or $y$ is a power of $x$. This is
not in the right form as it stands, but is easily translated; the \emph{power
graph} of an algebra $A$ has vertex set $A$, with $x$ joined to $y$ if either
$x\in\langle y\rangle$ or $y\in\langle x\rangle$. (Kelarev and Quinn themselves
studied power graphs and digraphs of semigroups~\cite{kq2}.)

\medskip\noindent
\textbf{The enhanced power graph} This graph was defined by Aalipour
\emph{et al.} \cite{aacns}, although Abdollahi and 
Hassanabadi~\cite{ah} had earlier considered the complementary graph. Two
elements $x$ and $y$ are joined if there is an element $z$ such that both $x$
and $y$ are powers of $z$; equivalently,
$\langle x,y\rangle=\langle z\rangle$ for some element $z$.

Translating this to arbitrary algebras gives an ambiguity, since the two
forms given are not equivalent in general: the first asserts that
$\langle x,y\rangle\le\langle z\rangle$ for some $z$. So there are two
possible definitions of the enhanced power graph:
\begin{itemize}
\item[(A)] $x\sim y$ if $\langle x,y\rangle$ has rank~$1$;
\item[(B)] $x\sim y$ if $x,y\in\langle z\rangle$ for some $z$.
\end{itemize}
I will call these the \emph{strict} and \emph{loose} definitions. They are
equivalent if the algebra is $1$-monotonic (for example, a group). For the
semigroup at the end of the last subsection, $x^2\sim x^3$ in the loose
definition (B) but not in the strict definition (A).

There are just a few places in this paper where it makes a difference; and, in
most of these, the loose definition gives nicer results. So I will adopt it
throughout.

\medskip\noindent
\textbf{The intersection power graph} This was defined more recently, by
Sudip Bera in 2018. The intersection power graph of a group $G$ has vertex set
$G$, with $x$ joined to $y$ if there is a non-identity element $z$ which is a
power of both $x$ and $y$. Also, by convention, the identity is joined to
everything. So, in an arbitrary algebra $A$, we join two elements $x$ and $y$
if $\langle x\rangle\cap\langle y\rangle$ properly contains $E(A)$, and join
elements of $E(A)$ to all other vertices.

\medskip\noindent
\textbf{The generating graph} Two elements $x$ and $y$ are joined in the 
\emph{generating graph} of $A$ if $\langle x,y\rangle=A$. (This graph has no
edges if $A$ is not $2$-generated. But for groups, it derives much of its
importance from the fact that any non-abelian finite simple group is
$2$-generated.)

\medskip\noindent
\textbf{The independence graph} This graph and the next were devised for
groups by Andrea Lucchini~\cite{lucchini} to get around the fact that not all
groups are $2$-generated. Two elements $x$ and $y$ of $A$ are joined in the
\emph{independence graph} if there is a minimal generating set for $A$
containing $x$ and $y$.

\medskip\noindent
\textbf{The rank graph} Similarly, $x$ and $y$ are joined in the
\emph{rank graph} if there is a generating set of minimum cardinality
containing $x$ and $y$.

\medskip

For each of these graphs, the automorphism group of the algebra acts as
automorphisms of the graph. Also, since they are all determined by the 
subalgebras of the algebra, we see that if two algebras (possibly with different
sets of operations) have a bijection between then mapping subalgebras to
subalgebras in both directions, then all of the above graph types give
isomorphic graphs on the two algebras.

The power graph of an algebra $A$ has the further property that induced
subgraph on a  subalgebra $B$ of $A$ is the power graph of $B$. Thus, given a
graph class $\mathcal{C}$, the class of algebras whose power graph
belongs to $\mathcal{C}$ is closed under taking subalgebras. The same is not
true for the enhanced power graph, but would be true if we used the strict
definition of enhanced power graph. In particular, the statement holds if
$A$ is $1$-monotonic.

By contrast, the generating graph has the very different property that the
induced subgraph on a proper subalgebra contains no edges, since two elements
in a proper subalgebra cannot generate the whole algebra.

\section{Results and open problems}

\subsection{Relations between graphs}

The following result is an elementary observation for groups, and indeed the
same arguments prove it for any class of algebras. A \emph{spanning subgraph}
of a graph $\Gamma$ is a graph $\Delta$ with the same vertex set as $\Gamma$
whose edge set is a subset of the edge set of $\Gamma$. The \emph{complement}
of a graph $\Gamma$ has the same vertex set as $\Gamma$; its edge set consists
of all pairs of distinct vertices which are non-edges in $\Gamma$.

\begin{theo} 
\begin{itemize}
\item[(a)] The power graph is a spanning subgraph of both the enhanced power
graph and the intersection power graph.
\item[(b)] The independence graph is a spanning subgraph of the complement of
the power graph.
\item[(c)] The rank graph is a spanning subgraph of the complement of the
enhanced power graph.
\end{itemize}
\end{theo}

I sketch the proofs. Suppose that $x$ and $y$ are joined in the power graph,
with (say) $x\in\langle y\rangle$).
\begin{itemize}
\item $x$ and $y$ lie in $\langle y\rangle$, so
$\langle x,y\rangle=\langle y\rangle$, and $x$ and $y$ are
joined in the enhanced power graph.
\item If $x\notin E(A)$, then
$\langle x\rangle\cap\langle y\rangle=\langle x\rangle\supset E(A)$; but if
$x\in E(A)$ then $x$ is joined to all other vertices in the intersection power
graph.
\item $x$ and $y$ cannot both be contained in a mimimal generating set, since
$x$ could be omitted from such a set.
\end{itemize}
Finally, if $x$ and $y$ are joined in the enhanced power graph, with (say)
$x,y\in\langle z\rangle$, then we can replace the pair $x$ and $y$
in a generating set by $z$ to obtain a smaller generating set.

Each of these results raises a natural question: For which algebras does
equality hold?

For parts (b) and (c), this question has been completely answered for groups by
Freedman, Lucchini, Nemmi and Roney-Dougal~\cite{flnr}. All groups whose power
graph is the complement of the independence graph, or whose enhanced power
graph is the complement of the rank graph, are determined in the cited paper;
and indeed all these groups are soluble (although the arguments require
detailed consideration of insoluble groups).

Nothing is known for other classes of algebras.

By contrast, for part (a), there is a structure theorem for such graphs:

\begin{theo}
Let $\Gamma$ be the power graph of an algebra $A$, and suppose that $\Gamma$
is equal to either the enhanced power graph or the intersection power graph
of $A$. Then $\Gamma$ is a cograph and a chordal graph.
\end{theo}

The proof of this theorem involves some preliminary results of interest in
their own right.

\begin{theo}
Let $(a,b,c,d)$ be an induced $4$-vertex path in the power graph of an 
algebra $A$. Then one of $\{a,c\}$ and $\{b,d\}$ is an edge in the enhanced
power graph of $A$, and the other is an edge in the intersection power graph.
\end{theo}

\begin{proof}
Anticipating the section on directed graphs, we put $x\to y$ if 
$y\in\langle x\rangle$. Now no element of $\{a,b,c,d\}$ is joined to the
other three, so none of these elements lies in $E(A)$. Also, we cannot have
$x\to y$ and $y\to x$ for $x,y\in\{a,b,c,d\}$, since this would imply that
$x$ and $y$ have the same neighbours. Further, we cannot have, for example,
$a\to b\to c$, since this would imply $a\to c$ and so $\{a,c\}$ would be an
edge. Thus we have $a\to b\gets c\to d$ or $a\gets b\to c\gets d$; without loss
of generality we can suppose that the first condition holds. But then
$a,c\in\langle b\rangle$, so $\{a,c\}$ is an edge of the enhanced power graph;
and $c\in\langle b\rangle\cap\langle d\rangle$, with $c\notin E(A)$, so
$\{b,d\}$ is an edge of the intersection power graph.
\end{proof}

\begin{cor}
If the power graph $\Gamma$ of $A$ is equal to either the enhanced power graph
or the intersection power graph, then $\Gamma$ is a cograph.
\end{cor}

The next, surprising, result was proved for the enhanced power graph of a
group by Bubboloni, Fumagalli and Praeger~\cite{bfp}. In fact, their argument
shows something much more general; it holds in any algebra, and its dual also
holds.

\begin{theo}
If either the enhanced power graph or the intersection power graph of an
algebra is a cograph, then that graph is also a chordal graph.
\end{theo}

\begin{proof}
Suppose that $\Gamma$ is the enhanced power graph of an algebra $A$ and is a
cograph but not a chordal graph. Then $\Gamma$ contains an induced cycle of 
length at least~$4$. It cannot have length greater than $4$, since such cycles
contain $4$-vertex paths. So it is a $4$-cycle $(a,b,c,d)$. 

Note that no two of $\langle a\rangle$, $\langle b\rangle$, $\langle c\rangle$,
and $\langle d\rangle$ can be equal or included one in another. For if
$a\in\langle c\rangle$ then $a$ and $c$ would be joined; and if
$a\in\langle b\rangle$ and $b,c\in\langle z\rangle$, then
$a,c\in\langle z\rangle$ and again $a$ and $c$ would be joined. All other
cases are similar.

Now $a$ and $b$ are joined, so $a,b\in\langle u\rangle$ for some $u$; choose
$u$ such that $\langle u\rangle$ is maximal with this property. By the preceding
paragraph, $\langle u\rangle$ is not equal to any of the subalgebras generated
by $a$, $b$, $c$ or $d$.

Suppose that $u$ is joined to $c$. Then there is an element $v$ such that $c$
and $u$ (and hence $a$ and $b$) are contained in $\langle v\rangle$. By the
maximality of $\langle u\rangle$, we must have
$\langle v\rangle=\langle u\rangle$, and so $a$ is joined to $c$, which is
false since the cycle is induced.

By an almost identical argument, $u$ is not joined to $d$. But it is joined to
$b$. So $(u,b,c,d)$ is a $4$-vertex path, contradicting the fact that
$\Gamma$ is a cograph.

The argument for the intersection power graph is almost identical but with
the inclusions reversed. We note that none of $a,b,c,d$ lies in $E(A)$,
and hence there is an element
$u\in\langle a\rangle\cap\langle b\rangle\setminus E(A)$; we choose $u$
so that, subject to this, $\langle u\rangle$ is minimal. 
\end{proof}

\paragraph{Question} Are there other results of this type for graphs on
algebras? For example, in a different context, Xuanlong Ma and I~\cite{mc}
showed that the commuting graph of a group is a split graph if and only if it
is a threshold graph; all groups with these properties are known (they are
abelian groups, and generalised dihedral groups of twice odd order).

\medskip

Groups whose power graph and enhanced power graph are equal have been
completely determined, but the argument doesn't use the above structure
theory.  Instead it uses another condition which extends to a wider class
of algebras (though not all).

An algebra $A$ has \emph{property MO} if the rank~$1$ subalgebras of any
rank~$1$ subalgebra of $A$ are totally ordered by inclusion. (The name is a
mnemonic for ``Monogenic subalgebras Ordered''. I have not found this condition
in the literature and don't know if it has another name.)

A finite group $G$ has property MO if and only if $G$ is an \emph{EPPO group},
meaning that every element of $G$ has prime power order. This is because the
cyclic subgroups of $C_n$ are totally ordered if and only if $n$ is a prime
power.

After the classification of soluble EPPO groups by
Higman~\cite{higman} in the 1950s, and of simple EPPO groups by
Suzuki~\cite{suzuki1,suzuki2} in the 1960s, the complete
classification of EPPO groups was given by Brandl~\cite{brandl} in 1981.

\begin{theo}
\begin{itemize}
\item[(a)] The power graph and enhanced power graph of an algebra $A$ are equal
if and only if $A$ has property MO.
\item[(b)] The power graph and enhanced power graph of a group $G$ are equal
if and only if $G$ is an EPPO group.
\end{itemize}
\end{theo}

\begin{proof}
(a) Suppose that $A$ is not MO. Then some rank~$1$ subalgebra
$\langle z\rangle$ contains a subalgebra of rank greater than $1$, which we
may suppose is $\langle x,y\rangle$; then $x$ and $y$ are joined in the
enhanced power graph but not in the power graph. Conversely, if $A$ is MO,
and $x,y\in\langle z\rangle$, then either $\langle x\rangle\le\langle y\rangle$
or the reverse, so $x$ and $y$ are joined in the power graph. (Note that one
direction requires the loose definition of enhanced power graph.)

(b) follows immediately.
\end{proof}

\medskip

The characterisation of algebras with power graph and intersection power graph
equal is not known even for groups. We know that this implies that the power
graph is a cograph; partial characterisations of
such groups are known, see~\cite{cmm}. Also, a group in which every
non-identity element has prime order has power graph equal to intersection
power graph. See~\cite{bc} for discussion.

\medskip

Here is another result known for groups~\cite{aaa,fmw} which extends to
arbitrary algebras.

\begin{theo}
The power graph of an algebra is perfect.
\end{theo}

\begin{proof}
We define a relation $\preceq$ on $A$ by the rule that $x\preceq y$ if
$\langle x\rangle\subseteq\langle y\rangle$. This is a partial preorder (a
reflexive and transitive relation), and the power graph is its comparability
graph (that is, distinct $x$ and $y$ are joined if $x\preceq y$ or
$y\preceq x$).

We can convert this partial preorder into a partial order as follows. Define
$x\equiv y$ if $x\preceq y$ and $y\preceq x$ (that is, if 
$\langle x\rangle=\langle y\rangle$). Now $\equiv$ is clearly an equivalence
relation; we simply modify $\preceq$ by totally ordering each equivalence
class. This doesn't change the comparability graph.

Now a theorem of Mirsky~\cite{mirsky} asserts that the comparability graph of
a partial order is perfect.
\end{proof}

\medskip
 
The enhanced power graph of a group is not in general perfect: indeed, enhanced
power graphs are \emph{universal}: every finite graph is an induced subgraph of
the enhanced power graph of some group~\cite[Theorem 5.5]{gog}. However, they
are close to power graphs in various ways. Here are a few.

\begin{itemize}
\item The enhanced power graph of a finite group is \emph{weakly perfect}:
this means that its clique number and chromatic number are equal~\cite{cp}.
\item The power graph and enhanced power graph of a group have the same
\emph{matching number} (this is the maximum size of a set of pairwise disjoint
edges)~\cite{css}.
\item There has been a study of the \emph{difference graph}, the graph whose
edge set is the difference of the edge sets of the enhanced power graph and
the power graph~\cite{bcdd}. This includes the construction of an interesting
sparse bipartite graph on 385 vertices with diameter and girth~$10$.
\end{itemize}
The arguments are very combinatorial, and it might be possible to extend them
to other types of algebras.

Another point where such extension might be possible involves cyclic groups.
For example, the clique numbers of the power graph and enhanced power graph
of a group are determined by the cyclic subgroups of the group: they are the
maximum values of these parameters over all cyclic subgroups. The enhanced
power graph of a cyclic group is complete, so its clique number is the group
order. But the clique number of the power graph of a cyclic group of order $n$
is an interesting number-theoretic function $f(n)$, closely related to Euler's
function $\phi(n)$. Indeed, we have
\[\phi(n)\le f(n)\le c\phi(n),\]
for some constant $c$; the best possible constant (the limit superior of the
ratio $f(n)/\phi(n)$) is roughly $2.6481017597\ldots$.

\subsection{Generating graph}

Generating graphs of groups have some important properties, and their study
may have introduced a new graph parameter: the \emph{spread} of a graph is the
maximum integer $s$ such that any $s$ vertices of the graph have a common
neighbour. Thus, a graph has spread~$1$ if and only if it has no isolated
vertices, and spread~$2$ if and only if any two vertices have a common
neighbour (so in particular, the graph is connected with diameter at most~$2$).

It was shown in \cite{bgk} that the spread of the generating graph of a
non-abelian finite simple group (with the identity deleted)
is at least~$1$. This was improved in \cite{bgh} to show that the spread is at
least~$2$, and indeed this paper gave a full characterisation of the finite
groups with spread at least~$2$.

Is any result possible about the arbitrary $2$-generated algebras with
spread at least~$2$?

\section{Zero-divisor graphs of posets}

The zero-divisor graph of a poset was introduced by Lu and Wu~\cite{lw} as
a tool for studying semigroups. In the hands of Vinayak Joshi and others, it
has potential to apply to a wider class of algebras. We give a very brief
account here.

Let $P$ be a poset with unique minimal element $0$. We denote by $S^\wedge$
the set $\{x\in P:(\forall s\in S)(x\le s)\}$. The element $a\in P$ is
a \emph{zero-divisor} if $a\ne 0$ and there exists $b\in P$ with $b\ne 0$
such that $\{a,b\}^\wedge=0$. The \emph{zero-divisor graph} of $P$ is the
graph whose vertices are the zero-divisors, with an edge $\{a,b\}$ if and
only if $\{a,b\}^\wedge=0$.

Given an algebra $A$, we define $P(A)$ to be the poset whose element set is $A$,
with $a\le b$ whenever $\langle a\rangle\supseteq\langle b\rangle$. (Note that
this is the dual of the poset induced by inclusion of $1$-generated
subalgebras.)
We add a new element $0$ defined to lie below all the elements of $A$.
Now we claim that the zero-divisor graph of this poset is the complement of
the enhanced power graph of $A$. For, if $\{x,y\}^\wedge=0$, then there is
no element $z\in P$ such that $x,y\in\langle z\rangle$, so $x$ and $y$ are
nonadjacent in the enhanced power graph; and conversely.

Devhare, Joshi and Lagrange~\cite{djl} give a sufficient condition on a poset
for the complement of its zero-divisor graph to be weakly perfect. While this
is not strong enough to deduce the result of \cite{cp} that the enhanced power
graph of a group is weakly perfect, it may be interesting to discover
algebras to which it does apply.

\section{Directed graphs}

\subsection{Power digraph}

The power graph of an algebra is more naturally regarded as a directed graph,
with an arc $x\to y$ if $y\in\langle x\rangle$. From this \emph{power digraph},
we obtain the power graph by ignoring directions, and the enhanced power graph
by the rule that $x\sim y$ if and only if there exists $z$ such that
$z\to x$ and $z\to y$.

In fact, for groups, these three objects contain the same information:

\begin{theo}
Let $\Gamma$ denote any one of the power digraph, power graph and
enhanced power graph. 
If $G_1$ and $G_2$ are groups for which $\Gamma(G_1)$ and $\Gamma(G_2)$
are isomorphic, then $\Delta(G_1)$ and $\Delta(G_2)$ are isomorphic, where
$\Delta$ denotes any  other of these three graph types.
\label{t:recon}
\end{theo}

The fact that the power graph determines the power digraph up to
isomorphism was proved in \cite{power,bp}, and the analogous fact for the
enhanced power graph in \cite{zbm}. Note the words ``up to isomorphism'':
the power graph of the cyclic group of order $6$ is the complete graph with
edges $\{2,3\}$ and $\{3,4\}$ deleted, so $0$, $1$ and $5$ are joined to all
other vertices; we can choose any one of them to be the identity, that is, a
sink in the power digraph, and the other two will be sources.

We note that the closed out-neighbourhood of a vertex $x$ in the power digraph
of any algebra is the subalgebra generated by $x$. So, for example, in a group,
the power digraph determines the orders of the elements. As we noted earlier,
the power digraph, with a loop at each vertex, is a partial preorder.

It is not known whether the analogue of Theorem~\ref{t:recon} holds in any
larger class of algebras.

\subsection{Endomorphism digraph}

Another directed graph associated with a group was studied recently in
\cite{endo}; its definition immediately extends to arbitrary algebras.
The \emph{endomorphism digraph} of an algebra $A$ has vertex set $A$,
with an arc $x\to y$ if there is an endomorphism of $A$ mapping $x$ to $y$.
The \emph{endomorphism graph} is obtained from the endomorphism digraph by
ignoring directions (and collapsing double edges to single edges).

It is easy to see that the endomorphism digraph, with a loop at each vertex,
is reflexive and transitive, that is, a partial preorder; so, as we saw for
the power graph, this implies that the endomorphism graph is perfect.

Many of the results in \cite{endo} are negative: the endomorphism graph does
not determine the endomorphism digraph, and the latter does not determine the
group (up to isomorphism); we can have $x\to y$ and $y\to x$ but with no
automorphism carrying $x$ to $y$. Among positive results for groups, which
may extend under some hypotheses to more general algebras, are the facts that,
if two groups $G$ and $H$ have coprime orders, then
\begin{itemize}
\item the endomorphism digraph of $G\times H$ is the Cartesian product of the
endomorphism digraphs of $G$ and $H$;
\item the endomorphism graph of $G\times H$ is the strong product of the
endomorphism graphs of $G$ and $H$.
\end{itemize}
The coprimality assumption implies that any endomorphism of $G\times H$ must
map $G$ and $H$ to themselves. Perhaps the generalisation will involve a
condition of this form.

\subsection{Equality}

In this section I investigate when the power and endomorphism digraphs are
equal. First, a general result about these digraphs. 

\begin{lemma}
Let $x$ and $y$ be elements of the finite algebra $A$ which are joined in both
the power graph and the endomorphism graph. Then either they are joined by 
arcs in both directions in the power and endomorphism digraphs, or they are
joined by single arcs in the same direction.
\end{lemma}

\begin{proof}
Suppose that there is an arc $x\to y$ in the power digraph and an arc $y\to x$
in the endomorphism graph. The first implies that $y\in\langle x\rangle$, so
$\langle y\rangle\subset\langle x\rangle$, whence
$|\langle x\rangle|\ge|\langle y\rangle|$. The second implies the existence of
an endomorphism $f$ mapping $y$ to $x$, and hence mapping $\langle y\rangle$
onto $\langle x\rangle$; so $|\langle x\rangle|\le|\langle y\rangle|$. Hence we
have equality: $\langle x\rangle=\langle y\rangle$.

This means, first, that $x\in\langle y\rangle$, so $y\to x$ in the power
digraph; second, the endomorphism $f$ maps $\langle x\rangle$ to itself, and
so induces an automorphism of this subalgebra carrying $y$ to $x$; then some
positive power of $f$ maps $x$ to $y$, so $x\to y$ in the endomorphism digraph.
(Note that the finiteness of $A$ is used here.)
\end{proof}

From this we obtain our first main result:

\begin{theo}
Let $A$ be a finite algebra. Then the power graph and endomorphism graph of
$A$ coincide if and only if the power digraph and endomorphism digraph
coincide.
\end{theo}

\begin{proof}
The ``if'' statement is clear from the definition, while the ``only if''
statement comes immediately from the lemma.
\end{proof}

We can translate this into a condition on the algebra. Recall that a
subalgebra $B$ of $A$ is \emph{fully invariant} if it is invariant under all
endomorphisms of $A$.

\begin{theo} Let $A$ be a finite algebra.
\begin{itemize}
\item The endomorphism digraph of $A$ is a spanning subgraph of the power
digraph if and only if every subalgebra of $A$ is fully invariant.
\item The power digraph of $A$ is a spanning subgraph of the endomorphism
digraph if and only if, for every $x,y\in A$ with $y\in\langle x\rangle$, there
is an endomorphism of $A$ mapping $x$ to $y$.
\end{itemize}
\end{theo}

\begin{proof}
For the first part, suppose first that endomorphism digraph is a spanning
subgraph of the power subgraph, and let $B$ be a subalgebra of $A$ and
choose $b\in B$. Then $b\to y$ is an arc of the power digraph if and only if
$y\in\langle b\rangle\le B$; so all endomorphisms map $y$ into $B$. Since
this holds for all $b\in B$, we see that $B$ is fully invariant. The converse
is proved by reversing the argument.

The second part is immediate from the definitions.
\end{proof}

Now we can decide which groups have the two graphs equal.

\begin{theo}
Let $G$ be a finite group. Then the power digraph and endomorphism digraph of
$G$ coincide if and only if $G$ is a cyclic group.
\end{theo}

\begin{proof}
If $G$ is cyclic, then all its endomorphisms are power maps, and so the 
digraphs are equal. (If $G=\langle g\rangle$ and the endomorphism $f$ maps
$g$ to $g^m$, then it maps $g^i$ to $g^{im}$ for all $i$.) 

Conversely, suppose the digraphs coincide. Then every subgroup of $G$ is
fully invariant. In particular, every subgroup is normal; in other words,
$G$ is a \emph{Dedekind group}. From the classification of Dedekind groups
\cite{dedekind}, we see that either $G$ is abelian or $G$ has the quaternion
group $Q_8$ as a direct factor. The latter case does not occur since $Q_8$
has an automorphism permuting the three subgroups of order $4$.

If $G$ is abelian but not cyclic, then we see from the classification of
finite abelian groups that we can write $G=A\times B$, where both $|A|$ and
$|B|$ are divisible by a prime $p$. Let $a$ and $b$ be elements of order $p$
in $A$ and $B$ respectively, and $H=\langle ab\rangle$. Then the endomorphism
mapping $A$ to the identity and fixing every element of $B$ maps $H$ to the
disjoint subgroup $\langle b\rangle$. So this case also doesn't occur.

Thus $G$ is cyclic.
\end{proof}

Mikhail Volkov has drawn my attention to two papers
\cite{maz1,maz2} by Ryszard Mazurek. These papers concern semigroups all of
whose endomorphisms are power maps. On the face of it, this is stronger than
asking that every endomorphism maps each element to some power of itself,
since they require it to be the same power in each case. However, as we have
seen above, the two properties are equivalent for groups.

Volkov also provided an example to show that they are not equivalent for all
semigroups. Let $S$ be the semigroup on the set $\{a,b,e\}$ with Cayley table
\begin{center}
\begin{tabular}{c|ccc}
   & $a$ & $b$ & $e$ \\
  \hline
  $a$ & $e$ & $b$ & $e$ \\
  $b$ & $b$ & $e$ & $b$ \\
  $e$ & $e$ & $b$ & $e$
\end{tabular}.
\end{center}
$S$ is the amalgamation of the 2-element group $\{b,e\}$ and the  2-element zero semigroup $\{a,e\}$ over the common idempotent $e$.

The endomorphisms of $S$ are exactly the following four maps:
\[
\begin{array}{c|ccc}
          & a & b & e\\ \hline
\varphi_1 & a & b & e\\
\varphi_2 & e & b & e\\
\varphi_3 & a & e & e\\
\varphi_4 & e & e & e
\end{array}
\]
We see that every endomorphism of $S$ maps each element to a power of itself. However, the endomorphism $\varphi_3$ is not a power map. Indeed, the equality $\varphi_3(x)=x^k$ fails for $x=b$ if $k=1$ and for $x=a$ if $k>1$.

\medskip

\noindent\textbf{Question} Can anything be said about semigroups in which any
endomorphism which maps every element $x$ to some power of $x$ is a power map?
(As far as I know, even the groups satisfying this are not classified.)

\section{Simplicial complexes}

A \emph{simplicial complex} is a family $\mathcal{C}$ of finite subsets of a
ground set $X$ which is closed under taking subsets. The subsets are called
\emph{simplices}. We assume that every singleton of $X$ is a simplex; other
elements of $X$ are isolated and can be discarded.

Simplicial complexes have a topological interpretation, according to which each
simplex has a (geometric) dimension which is one less than its cardinality: so
elements of $X$ are \emph{vertices}, $2$-element subsets are \emph{edges}, and
so on. Note that the vertices and edges form a graph, known as the
\textit{$1$-skeleton} of~$X$.

Various simplicial complexes can be defined on an algebra; I focus here on
two examples studied for groups in \cite{simplicial}, known as the
independence and the strong independence complexes. The definition of these
extends without change to arbitrary algebras, and the arguments here follow
closely those in \cite{simplicial}.

\subsection{The independence complex}

A subset $S$ of an algebra $A$ is \emph{independent} if
$s\notin\langle S\setminus\{s\}\rangle$ for all $s\in S$. The
\emph{independence complex} of $A$ is the set of of all the
independent subsets of $A$. With this definition, the elements of $E(A)$ are
not contained in any simplex, since if $c\in E(A)$ then
$c\in\langle\emptyset\rangle$. So we remove these.

\begin{prop}
The independence complex is a simplicial complex; its vertices are the
elements of $A\setminus E(A)$. Its $1$-skeleton is the induced subgraph of the
complement of the power graph of $A$ on this set.
\end{prop}

\begin{proof}
Suppose that $S$ is independent and $T\subseteq S$. If
$t\in\langle T\setminus\{t\}\rangle$, then $t\in\langle S\setminus\{t\}\rangle$,
contradicting the independence of $S$; so $T$ is independent.

A $1$-element set $\{x\}$ is independent if and only if $x\notin E(A)$.

A $2$-element set fails to be independent if and only if one of its elements
lies in the algebra generated by the other; that is, they are joined in the
power graph. \qed
\end{proof}

\subsection{The strong independence complex}

The concept of independence in a group has been investigated by a number of
authors. However, I know no literature on the next concept.

A subset $S$ of an algebra $A$ is called \emph{strongly independent} if no
subalgebra of $A$ containing $S$ has fewer than $|S|$ generators.
The \emph{strong independence complex} of $A$ is the complex whose simplexes
are the strongly independent sets.

\begin{prop}
The strong independence complex is a simplicial complex, whose vertex set is
$A\setminus E(A)$. Its $1$-skeleton is the induced subgraph of the complement
of the enhanced power graph of $A$ on this set.
\end{prop}

\begin{proof}
Let $S$ be strongly independent and $T\subseteq S$. Suppose that $T$ is not
strongly independent; so there exists a subalgebra $B$ containing $T$ which
is generated by a set $U$ with $|U|<|T|$. Let $C$ be the subalgebra generated
by $(S\setminus T)\cup U$. Then $C$ contains $T$, and hence $S$; and
the generating set of $C$ is smaller than $S$, contradicting strong
independence of $S$.

It is clear that a singleton $\{x\}$ is strongly independent if and only if
$x\notin E(A)$.

For the last part, suppose that $x$ and $y$ are adjacent in the enhanced power
graph, so that $x,y\in\langle z\rangle$ for some $z$; thus
$\{x,y\}$ is not strongly independent. Conversely, if $\{x,y\}$ is not
strongly independent, then $x,y\in\langle z\rangle$ for some $z$, so $x$ and
$y$ are  adjacent in the enhanced power graph. (Note that this argument uses
the loose definition of enhanced power graph.) 
\end{proof}

Clearly a strongly independent set is independent. This prompts a further
question:
\begin{quote}
For which algebras do the notions of independence and strong independence
coincide?
\end{quote}

The answer is known for groups: the classification of such groups was given
by Lucchini and Stanojkovski~\cite{ls}. The nilpotent examples are just the
\emph{monotone $p$-groups}, those $p$-groups $G$ for which $K\le H\le G$
implies $r(K)\le r(H)$. The non-nilpotent examples are Frobenius groups, whose
Frobenius kernel is abelian and Frobenius complement cyclic of prime power
order (with a further numerical restriction).

Lucchini and Stanojkovski also prove the following theorem for groups;
it would be interesting to know whether a similar result holds for wider
classes of algebras.

\begin{theo}
If $G$ is a finite group whose independence complex is isomorphic to that of
an abelian group, then $G$ is nilpotent, and its Sylow subgroups are modular
and non-Hamiltonian.
\end{theo}

\medskip

I end this section with a remark and a question. The homology groups of the
independence and strong independence complexes of an algebra $A$ afford
representations of the automorphism group of $A$. What can be learned about
$A$ by studying these representations? (See \cite{vick} for homology groups.)

\section{Independence algebras}

\emph{Independence algebras} form an important class of algebras with rich
endomorphism monoids: see \cite{abckk} for a recent survey. These are algebras
in which
\begin{itemize}
\item the subalgebras have the exchange property (that is, are flats of a
matroid), so all bases (minimal generating sets) have the same cardinality;
\item any map from a basis into the algebra extends uniquely to an endomorphism
of the algebra.
\end{itemize} 
Note that precise details of the operations are not required; the class of
independence algebras is closed under \emph{SE-equivalence}, that is, bijection
preserving subalgebras and endomorphisms. Independence algebras are classified
uo to SE-equivalence: they are based on groups acting on sets, ``blown up'' by
Boolean algebras; vector spaces; affine spaces; and sharply $2$-transitive
groups.  Clearly these algebras are monotonic, indeed strictly monotonic
(that is, $B_1<B_2$ implies $r(B_1)<r(B_2)$) abd have the MO property.

The graphs and digraphs defined here are rather trivial for
independence algebras, though the simplicial complexes are more interesting.

\subsection{Graphs and digraphs}

Let $A$ be an independence algebra. In the corresponding matroid, the elements
of $E(A)$ are the loops, and the sets
$\langle x\rangle^-=\langle x\rangle\setminus E(A)$ are the classes of
parallel elements. So the power digraph has all arcs between elements of
$E(A)$, and all arcs $x,y$ with $y\in\langle x\rangle$. Thus the power graph,
enhanced power graph, and intersection power graph are all equal, and are
``sunflowers'' consisting of complete graphs on the sets $\langle x\rangle$,
any two sharing the vertices of $E(A)$.

The independence and rank graphs are equal, and are the complement of the power
graph: that is, complete multipartite on $A\setminus E(A)$ with partite sets
$\langle x\rangle^-$.

The generating graph is as follows:
\begin{itemize}
\item null, if $r(A)>2$;
\item equal to the independence graph, if $r(A)=2$;
\item complete on $A$ with edges in $E(A)$ deleted, if $r(A)=1$.
\end{itemize}

Any element of $A\setminus E(A)$ is part of a basis, whereas elements of
$E(A)$ are fixed by all endomorphisms. So the arcs of the endomorphism 
digraph are all those with source in $A\setminus E(A)$.

\subsection{Simplicial complexes}

For independence algebras, the independence complex and strong independence
complex coincide, and the simplices are the independent sets in the matroid
associated with the algebra. So they determine the subalgebra structure. Do
they determine the endomorphisms?

Another question that suggests itself is: if the independence complex of an
algebra is a matroid, is the algebra an independence algebra?

Negative examples to both questions can be obtained as follows. Let $Q$ be a
\emph{quasigroup} of order $n$, a set with a binary operation whose operation
table is a Latin square. We construct a new algebra $Q^{(1)}$ with $n$ unary
operations, these being the right multiplications in $Q$ (the maps
$\rho_a:x\mapsto xa$). This algebra has just two subalgebras, the empty set
and the whole of $Q$, so its independence complex is a (trivial) matroid.
Now the endomorphisms are all automorphisms, and they are the permutations
which commute with all the right multiplications. It is well known that, if
this group is transitive, then $Q$ is a group, and the automorphism group of
$Q^{(1)}$ is the set of left multiplications in $Q$. Thus, we have an
independence algebra if and only if $Q$ is a group; and different groups give
different independence algebras.

The two questions above may  be more interesting for algebras of larger rank.

\subsection*{Acknowledgement}

I am grateful to Daniela Bubboloni for sending me a copy of her paper and for
helpful comments which have improved the presentation of this paper, and to
Mikhail Volkov for drawing my attention to the papers of Mazurek, for providing
an example, and for several further comments.

\end{document}